\documentclass[reqno]{amsart}
\usepackage[foot]{amsaddr}
\usepackage[frozencache,cachedir=.,newfloat=true]{minted}
\usepackage{blkarray}
\usepackage[hidelinks]{hyperref}
\usepackage{url}
\usepackage{amssymb}
\usepackage{amscd}
\usepackage{amsthm}
\usepackage{setspace}
\usepackage{enumerate}
\usepackage{graphicx}
\usepackage{mathtools}
\usepackage{subfigure}
\usepackage{siunitx}
\usepackage{tikz-cd}
\usepackage{bm}
\numberwithin{equation}{section}
\usepackage[obeyFinal,textsize=footnotesize]{todonotes}

\usepackage{tcolorbox}

\usepackage{mathtools}
\usepackage[tableposition=top]{caption}
\usepackage{booktabs,dcolumn}

\title[New Upper Bound for Growth Factor in Gaussian Elimination]{A New Upper Bound For the Growth Factor in Gaussian Elimination with Complete Pivoting}
\author{Ankit Bisain}
\address{\scriptsize Department of Mathematics, Massachusetts Institute of Technology, Cambridge, MA, 02139 USA.
}
\email{ankitb12@mit.edu}
\author{Alan Edelman}
\email{edelman@mit.edu}
\author{John Urschel}
\email{urschel@mit.edu}

\newtheorem{theorem}{Theorem}[section]
\newtheorem{lemma}[theorem]{Lemma}

\newtheorem{proposition}[theorem]{Proposition}

\begin{document}

\begin{abstract}
    The growth factor in Gaussian elimination measures how large the entries of an LU factorization can be relative to the entries of the original matrix. It is a key parameter in error estimates, and one of the most fundamental topics in numerical analysis. We produce an upper bound of $n^{0.2079 \ln n +0.91}$ for the growth factor in Gaussian elimination with complete pivoting -- the first improvement upon Wilkinson's original 1961 bound of $2 \, n ^{0.25\ln n +0.5}$. 
\end{abstract}

\maketitle

\section{Introduction}

The solution of a linear system $Ax = b$ is one of the oldest problems in mathematics. One of the most fundamental and important techniques for solving a linear system is Gaussian elimination, in which a matrix is factored into the product of a lower and upper triangular matrix. Given an $n \times n$ matrix $A$, Gaussian elimination performs a sequence of rank-one transformations, resulting in the sequence of matrices $A^{(k)}\in \mathbb{C}^{k \times k}$ for $k$ equals $n$ to $1$, satisfying
\begin{equation}\label{eqn:block_def}
A^{(k)} = M^{(2,2)} - M^{(2,1)} [M^{(1,1)}]^{-1} M^{(1,2)},  \qquad \text{where } A = \begin{blockarray}{ccc}
         {\scriptstyle n-k} & {\scriptstyle k} &  \\
       \begin{block}{[cc]c}
         M^{(1,1)} & M^{(1,2)} &{ \scriptstyle n-k} \\ 
          M^{(2,1)} & M^{(2,2)} & {\scriptstyle k} \\
       \end{block}
     \end{blockarray}. 
     \end{equation}
The resulting LU factorization of $A$ is encoded by the first row and column of each of the iterates $A^{(k)}$, $k = 1,...,n$. Not all matrices have an LU factorization, and a permutation of the rows (or columns) of the matrix may be required. In addition, performing computations in finite precision can elicit issues due to round-off error. The error due to rounding in Gaussian elimination for a matrix $A$ in some fixed precision is controlled by the growth factor of the Gaussian elimination algorithm, defined by
\[g(A):= \frac{\max_{k} \|A^{(k)}\|_{\max}}{ \|A\|_{\max}},\]
where $\|\cdot\|_{\max}$ is the entry-wise matrix infinity norm (see \cite[Theorem 3.3.1]{golub2013matrix} for details\footnote{Here, we study the growth factor in exact arithmetic, while it is growth factor in floating point arithmetic that occurs in error estimates. The worst-case behavior of these two quantities is very similar (see \cite[Theorem 1.5]{edelman2024some}).}). For this reason, understanding the growth factor is of both theoretical and practical importance. Complete pivoting, famously referred to as ``customary" by von Neumann \cite{von1947numerical}, is a strategy for permuting the rows and columns of a matrix so that, at each step, the pivot (the top-left entry of $A^{(k)}$) is the largest magnitude entry of $A^{(k)}$. Complete pivoting remains the premier theoretical permutation strategy for performing Gaussian elimination. Despite its popularity, the worst-case behavior of the growth factor under complete pivoting is poorly understood. This is in stark contrast to partial pivoting, an alternative strategy which is incredibly popular in practice but known to be horribly unstable in the worst case (see Wilkinson's 1965 {\it The Algebraic Eigenvalue Problem} \cite[pg. 212]{Wilkinson1965AEP}). Here, we focus exclusively on the pure mathematical problem of the growth factor under complete pivoting. For the reader interested in the engineering aspects of solving a linear system, a detailed discussion of the role of the growth factor and complete vs partial pivoting in modern practice is provided in Appendix \ref{app:a}. 


\subsection{Historical Overview and Relevant Results} In their seminal 1947 paper {\it Numerical Inverting of Matrices of High Order}, von Neumann and Goldstine studied the stability of Gaussian elimination with complete pivoting \cite{von1947numerical}. This work was motivated by their development of the first stored-program digital computer and desire to understand the effect of rounding in computations on it \cite{meyer}. Goldstine later wrote:
\begin{quote}Indeed, von Neumann and I chose this topic for the first modern paper on numerical analysis ever written precisely because we viewed the topic as being absolutely basic to numerical mathematics \cite{goldstine1993computer}.
\end{quote}
However, it was not until Wilkinson's 1961 paper {\it Error Analysis of Direct Methods of Matrix Inversion} that a more rigorous analysis of the backward error in Gaussian elimination due to rounding errors occurred. Indeed, Wilkinson was the first to fully recognize the dependence of this error on the growth factor. Let $g_n(\mathbb{R})$ and $g_n(\mathbb{C})$ denote the maximum growth factor under complete pivoting over all non-singular $n \times n$ real and complex matrices, respectively. Wilkinson produced a bound for the growth factor under complete pivoting using only Hadamard's inequality \cite[Equation 4.15]{wilkinson1961error}:
\begin{equation}\label{eqn:wilk}
g_n(\mathbb{C}) \le \sqrt{n} \big( 2 \; 3^{1/2} \; ... \; n^{1/(n-1)} \big)^{1/2} \le 2 \sqrt{n} \, n^{\ln (n)/4},
\end{equation}
where the second inequality is asymptotically tight. This estimate was considered extremely pessimistic, with Wilkinson himself noting that ``no matrix has been encountered for which [the growth factor for complete pivoting] was as large as 8 \cite{wilkinson1961error}." A conjecture that the growth factor for complete pivoting of a real $n \times n$ matrix was at most $n$ was eventually formed (see \cite[Section 1.1]{edelman2024some} for a detailed discussion of the conjecture and its possible mis-attribution to both Cryer and Wilkinson). According to Higham, in his now-classic text {\it Accuracy and Stability of Numerical Algorithms}, this conjecture ``became one of the most famous open problems in numerical analysis, and has been investigated by many mathematicians \cite[pg. 181]{higham2002accuracy}." Many researchers attempted to upper bound the growth factor, with $g_n(\mathbb{R})$ computed exactly for $n = 1,2,3,4$ and shown to be strictly less than five for $n = 5$ (see the works of Tornheim  \cite{tornheim1964maximum,tornheim1965maximum,
tornheim1969maximum,tornheim1970maximum}, Cryer
\cite{cryer1968pivot}, and Cohen \cite{cohen1974note} for details). However, no progress was made on improving the bound for arbitrary $n$. Many years later, in 1991, Gould found a $13 \times 13$ matrix with growth factor larger than $13$ in finite precision \cite{gould1991growth} (extended to exact arithmetic by Edelman \cite{edelman1992complete}), providing a counterexample to the conjecture for $n = 13$. Recently, Edelman and Urschel 
 improved the best-known lower bounds for all $n>8$ and showed that 
\[g_n(\mathbb{R}) \ge 1.0045\, n \;  \text{ for all } n \ge 11, \quad\text{ and } \quad \limsup_n \big(g_n(\mathbb{R}) /n\big) \ge 3.317, \]
thus disproving the aforementioned conjecture for all $n \ge 11$ by a multiplicative factor \cite{edelman2024some}. However, for the upper bound, to date no improvement has been made to Wilkinson's bound.

\subsection{Our Contributions} In this work, we improve Wilkinson's upper bound by an exponential constant, the first improvement in over sixty years. In particular, we prove the following theorem, obtaining a leading exponential constant of $\tfrac{1}{2[2+(2-\sqrt{2})\ln 2]} \approx 0.20781$.

\begin{theorem}\label{thm:main} 
$g_n(\mathbb{C}) \le  n^{\tfrac{\ln n}{2[2+(2-\sqrt{2})\ln 2]} +0.91}.$
\end{theorem}

Our proof consists of four parts: 
\begin{enumerate}
    \item A Generalized Hadamard's inequality:  We prove a tighter version of Hadamard's famous inequality for matrices
with a large low-rank component. This generalization allows for a more sophisticated
analysis of the iterates of Gaussian elimination, providing additional constraints on the pivots of a matrix. (Subsection \ref{sub:det})
\item An Improved Optimization Problem: Applying the improved Hadamard inequality produces an optimization problem that can be considered a refinement of the optimization problem associated with Wilkinson's proof. Unfortunately, this refinement is no longer linear upon a logarithmic transformation. (Subsection \ref{sub:improvedopt})
\item From Non-Linear to Linear: We relax the logarithmic transformation of our optimization problem to a linear program, and prove that the optimal value of our relaxation has the same asymptotic behavior. (Subsection \ref{sub:improvedlp})
\item An Asymptotic Analysis: Finally, we analyze the asymptotic behavior of our linear program by converting it into a continuous program and applying a duality argument, thus producing the improved bound in Theorem \ref{thm:main}. (Section \ref{sec:lpbound})
\end{enumerate}

Our proof considers the same information regarding the underlying matrix as Wilkinson's original bound, using only the pivots at each step of elimination, and reveals further structure regarding the relationships between them. Our technique increases our understanding of the mathematical forces that constrain entries from increasing in size during Gaussian elimination, by illustrating the trade-off between having entries that grow quickly in size and having a matrix of large numerical rank (e.g., many large singular values). Improved estimates on the explicit constants in Theorem \ref{thm:main} can be obtained through a refinement of the techniques presented herein. However, tight estimates on the maximum growth factor will likely require further information regarding matrix entries. 

The techniques employed here can likely be used to improve upper bounds for the growth factor problem under other pivoting strategies (e.g., rook pivoting, threshold pivoting, etc.). We leave this natural extension to the interested reader.

\subsection{Notation and Basic Observations}

Recall that $A^{(k)}$, defined in Equation \ref{eqn:block_def}, denotes the $k \times k$ matrix resulting from the $(n-k)^{th}$ step of Gaussian elimination, and let $p_k$ denote the pivot of $A^{(k)}$ for $k = 1,\ldots,n$. Let $\langle \cdot, \cdot \rangle_F$ and $\| \cdot \|_F$ denote the Frobenius inner product and norm. Gaussian elimination under complete pivoting permutes the rows and columns of a matrix $A$ so that $p_k = \|A^{(k)}\|_{\max}$ for all $k$. Without loss of generality, we may assume $A$ is already completely pivoted, removing the need for pivoting in analysis. For complete pivoting, the growth factor is given by $\max_k p_k / p_n$, as $p_k = \|A^{(k)}\|_{\max}$, $k = 1, \ldots, n$, and $p_n = \|A\|_{\max}$. Because we are interested in the maximum growth factor over all $n \times n$ matrices, it suffices to consider the maximum value of $p_1/p_n$ \cite[Proposition 2.9]{day1988growth}. 

\section{Wilkinson's Bound Viewed as a Linear Program} \label{sec:wilk}

The proof of Wilkinson's 1961 bound is incredibly short, requiring one page of mathematics and using only Hadamard's inequality applied to the matrix iterates $A^{(k)} \in \mathbb{C}^{k \times k}$ of Gaussian elimination and the well-known fact that the product of pivots for a matrix equals its determinant. Hadamard's inequality, that the modulus of the determinant of a matrix is at most the product of the two-norm of its columns, implies that 
\begin{equation}\label{ineq:wilk_const}
    \prod_{i=1}^k p_i = \det(A^{(k)}) \le k^{k/2} |A^{(k)}|_\infty^k = k^{k/2} p_k^k.
\end{equation} 
The maximum $k^{th}$ pivot, viewed as a function of $k$, is non-decreasing, and so the maximum value of $p_1/p_n$ under these constraints provides an upper bound for the maximum growth factor:

\begin{tcolorbox}
\begin{equation}\label{prog:wilko}
\begin{gathered}
\textbf{Wilkinson's Optimization  Problem} 
 \\[1ex]
\begin{array}{cll}
\textrm{max} \quad & p_1/p_n & \\
\textrm{s.t.} \quad & \prod_{i = 1}^k p_i \le  k^{k/2}p_k^k   &\text{for } k = 1,...,n.\\
\end{array}
\end{gathered}
\end{equation}
\end{tcolorbox}
\vspace{1 mm}

\noindent Performing the transformation $q_k = \ln (p_k)$ for $k =1,...,n$ produces the linear program (LP): 

\begin{tcolorbox}
\begin{equation}\label{prog:wilkl}
\begin{gathered}
\textbf{Wilkinson's Linear Program} 
 \\[1ex]
\begin{array}{cll}
\textrm{max} \quad & q_1 - q_n & \\
\textrm{s.t.} \quad & \sum_{i = 1}^k q_i \le \frac{k}{2} \ln k + k \, q_k   &\text{for } k = 1,...,n.\\
\end{array}
\end{gathered}
\end{equation}
\end{tcolorbox}
\vspace{1 mm}

\noindent Wilkinson's proof, though never stated in the context of linear programming, can be viewed as a simple LP duality argument for the linear program maximize  $c^T x$ subject to $A x \le b$, where $x = (q_1,\ldots,q_n)^T$, $c = (1,0,\ldots,0,-1)^T$, and $Ax \le b$ encodes the inequalities of Linear Program \ref{prog:wilkl}: 
\[ A x  = \begin{pmatrix}  -1 & &  &  &  \\[.25ex] 1 & -1 & & &  \\[.25ex] 1 & 1& -2 & & \\[.25ex] \vdots & \vdots & \ddots & \ddots& \\[.25ex] 1 &1 & \cdots &1 &-(n-1)\end{pmatrix} \begin{pmatrix} q_1 \\[.25ex] q_2 \\[.25ex] q_3 \\[.25ex] \vdots \\[.25ex] q_n \end{pmatrix}\le \begin{pmatrix} 0 \\[.25ex] \tfrac{2}{2} \ln 2 \\[.25ex] \tfrac{3}{2} \ln 3 \\[.25ex] \vdots \\[.25ex] \tfrac{n}{2} \ln n \end{pmatrix} =  b.\]
We note that the additional constraint $q_1 \ge 0$ plays no role, as the feasible region of Program \ref{prog:wilkl} is shift-independent. The matrix $A$ has an easily computable inverse with $A^{-1}_{i1} = -1$ for $i = 1,...,n$, $A^{-1}_{ii} = - \tfrac{1}{i-1}$ for $i = 2,...,n$, and $A^{-1}_{ij} = - \tfrac{1}{j(j-1)}$ for $i>j$. The quantity
\[[A^{-1}]^T c  =  \begin{pmatrix}  -1 &-1 & -1 & \cdots & -1  \\[.25ex] & -1 & -\tfrac{1}{2} & \cdots & -\tfrac{1}{2}  \\[.25ex] &  & -\tfrac{1}{2} & &\vdots \\[.25ex] &  & & \ddots&-\tfrac{1}{(n-2)(n-1)} \\[.25ex]  & &  & &-\tfrac{1}{n-1} \end{pmatrix}\begin{pmatrix} 1 \\[.25ex] 0 \\[.25ex] \vdots \\[.25ex] 0\\[.25ex] -1 \end{pmatrix} = \begin{pmatrix} 0 \\[.25ex] \tfrac{1}{2} \\[.25ex] \vdots \\[.25ex] \tfrac{1}{(n-2)(n-1)} \\[.25ex] \tfrac{1}{n-1} \end{pmatrix} \]
is entry-wise non-negative, implying Wilkinson's bound
\[ q_1 - q_n = \big( [A^{-1}]^T c \big)^T A x  \le \big( [A^{-1}]^T c \big)^T b = \frac{1}{2} \bigg[ \ln n + \sum_{k=2}^n \frac{\ln k}{k-1} \bigg].\]
This bound is the exact solution to Program \ref{prog:wilkl}, evidenced by the matching feasible point $x = A^{-1} b$. The ease with which the optimal point of the dual program can be obtained is due to the simple structure of the constraints. Our improved linear program, described in Subsection \ref{sub:improvedlp}, has a more complicated set of constraints, requiring a more complex duality argument (given in Section \ref{sec:lpbound}).

This same argument also immediately produces bounds for the geometric mean growth factor of the iterates $A^{(k)}$, a key quantity in our proof of Theorem \ref{thm:main} that may be of independent interest. Indeed, the quantity $\tfrac{1}{n} \sum_{k=1}^{n} (q_1 - q_k)$ can be upper bounded by analyzing the linear program:
\vspace{.5 mm}

\begin{tcolorbox}
\begin{equation}\label{prog:geomean}
\begin{gathered}
\textbf{Geometric Mean Growth LP} 
 \\[1ex]
\begin{array}{cll}
\textrm{max} \quad & \frac{1}{n} \sum_{k = 1}^n (q_1 - q_k) & \\[.25ex]
\textrm{s.t.} \quad & \sum_{i = 1}^k q_i \le \frac{k}{2} \ln k + k \, q_k   &\text{for } k = 1,...,n.\\
\end{array}
\end{gathered}
\end{equation}
\end{tcolorbox}
\vspace{1 mm}
\noindent The constraints of this linear program are identical to those of Program \ref{prog:wilkl}. The only difference is in the objective; here we have $c = \big(\tfrac{n-1}{n},-\tfrac{1}{n},...,-\tfrac{1}{n}\big)^T$. Nevertheless, the quantity
\[[A^{-1}]^T c  = 
\begin{pmatrix*}  -1 &-1 &   -1  & \cdots & -1  \\[.25ex] & -1 & -\tfrac{1}{2} & \cdots & -\tfrac{1}{2}  \\[.25ex] &  & -\tfrac{1}{2} & &\vdots \\[.25ex] &  & & \ddots&-\tfrac{1}{(n-2)(n-1)} \\[.25ex]  & &  & &-\tfrac{1}{n-1} \end{pmatrix*}
\begin{pmatrix*} \tfrac{n-1}{n} \\[.25ex] -\tfrac{1}{n} \\[.25ex] \vdots \\[.25ex] -\tfrac{1}{n}\\[.25ex] -\tfrac{1}{n} \end{pmatrix*} = \begin{pmatrix} 0 \\[.25ex] \tfrac{1}{2} \\[.25ex] \vdots \\[.25ex] \tfrac{1}{(n-2)(n-1)} \\[.25ex] \tfrac{1}{(n-1)n} \end{pmatrix} \]
is entry-wise non-negative, implying the bound
\begin{equation}\label{ineq:avg_bound} \frac{1}{n} \sum_{k = 1}^n (q_1 - q_k) = \big( [A^{-1}]^T c \big)^T A x  \le \big( [A^{-1}]^T c \big)^T b = \frac{1}{2}\sum_{k=2}^n \frac{\ln k}{k-1}  \le \frac{\ln^2 n}{4} + \ln 2,
\end{equation}
or, in terms of the original pivots, 
\[ \Bigg[\prod_{k=1}^n \frac{p_1}{p_k} \Bigg]^{\tfrac{1}{n}} = \frac{p_1}{\left( \prod_{k=1}^n p_k \right)^{1/n}} \le 2 n^{\tfrac{1}{4} \ln n}. \]
This can be easily generalized further to any weighted average $\sum_{k=1}^n w_k (q_1 - q_k)$ of the logarithmic growth factors.

\section{An Improved Linear Program}

In this section, we produce additional constraints that the pivots must satisfy by generalizing Hadamard's inequality for matrices with a large low-rank component. These constraints, applied to the matrix $A^{(k)}$ (viewed as a sub-matrix of $A^{(k+\ell)}$ plus a rank $\ell$ matrix), lead to a new linear program with optimal value at most $0.2079 \ln^2 n + O(\ln n)$, the first improvement to the exponential constant of $0.25$ in Wilkinson's bound (Inequality \ref{eqn:wilk}).

\subsection{Improved Determinant Bounds}\label{sub:det}

First, we recall the following basic proposition, itself a corollary of \cite[Theorem 1]{li1995determinant}. \footnote{Proposition \ref{prop:sing_bound} also follows from applying standard determinant bounds for Hermitian matrices \cite[Theorem VI.7.1]{bhatia2013matrix} to $\begin{psmallmatrix} 0 & A \\ A^* & 0 \end{psmallmatrix}$ and $\begin{psmallmatrix} 0 & B \\ B^* & 0 \end{psmallmatrix}$, and using the following well-known rearrangement inequality: for any $a_1 \ge ... \ge a_n \ge 0$, $b_1 \ge ... \ge b_n \ge 0$, and $\pi \in \mathrm{S}_n$, $\prod_{i=1}^n (a_i + b_{\pi(i)}) \le \prod_{i=1}^n (a_i + b_{n-i+1})$.}

\begin{proposition}\label{prop:sing_bound}
$|\det(A+B)| \le  \prod_{i=1}^n \big(\sigma_i(A)+ \sigma_{n-i+1}(B)\big)$ for all $A, B \in \mathbb{C}^{n \times n}$, where $\sigma_1(A) \ge ... \ge \sigma_n(A)$ and $\sigma_1(B) \ge ... \ge \sigma_n(B)$ are the singular values of $A$ and $B$.
\end{proposition}

Next, we produce a generalized version of Hadamard's inequality for matrices with a large low-rank component. Here and in what follows, we use the convention that $0^0 = 1$.

\begin{lemma}\label{lm:det_bound}
    Let $A,B \in \mathbb{C}^{n\times n}$ with $\|A\|_F \le n$ and $\|B\|_F \le C n$, and $\mathrm{rank}(B)\le \ell$. Then
    \[ |\det(A+B)|\le \frac{n^n}{(n-\ell)^{\frac{n-\ell}{2}}\ell^{\frac{\ell}{2}}} \big(1+C\big)^\ell. \]
\end{lemma}

\begin{proof}
Let $0<\ell<n$, and $\sigma_1(A) \ge ... \ge \sigma_n(A)$ and $\sigma_1(B) \ge ... \ge \sigma_n(B)$ denote the singular values of $A$ and $B$. By Proposition \ref{prop:sing_bound},
\begin{align*}
     |\det(A+B)| &\le \bigg(\prod_{i=1}^{n-\ell} \sigma_i(A) \bigg) \prod_{j=1}^\ell\big( \sigma_j(B) + \sigma_{n-j+1}(A) \big) \\
     &\le \bigg( \frac{1}{n-\ell} \sum_{i=1}^{n-\ell} \sigma^2_i(A) \bigg)^{\frac{n-\ell}{2}} \bigg(\frac{1}{\ell} \sum_{j=1}^\ell \sigma_j(B) + \frac{1}{\ell} \sum_{j=1}^\ell\sigma_{n-j+1}(A)  \bigg)^\ell\\
     &\le \bigg( \frac{1}{n-\ell} \sum_{i=1}^{n-\ell} \sigma^2_i(A) \bigg)^{\frac{n-\ell}{2}} \bigg(\frac{1}{\ell^{\frac{1}{2}}} \bigg[\sum_{j=1}^\ell \sigma^2_j(B) \bigg]^{\frac{1}{2}} + \frac{1}{\ell^{\frac{1}{2}}}\bigg[ \sum_{j=1}^\ell\sigma_{n-j+1}^2(A) \bigg]^{\frac{1}{2}} \bigg)^\ell\\
     &\le \bigg( \frac{n^2}{n-\ell}\bigg)^{\frac{n-\ell}{2}} \bigg(\frac{Cn}{\ell^{\frac{1}{2}}}  + \frac{n}{\ell^{\frac{1}{2}}}\bigg)^\ell \\
     &= \frac{n^n}{(n-\ell)^{\frac{n-\ell}{2}}\ell^{\frac{\ell}{2}}}(1+C)^\ell,
\end{align*}
where we have used the AM-GM inequality in the second inequality and Cauchy-Schwarz in the third. The result for the cases $\ell=0$ and $\ell =n$ follows from gently modified versions of the same analysis. 
\end{proof}

We note that, when $\ell = 0$, Lemma \ref{lm:det_bound} implies the well-known corollary $|\det(A)| \le n^{n/2} |A|_\infty$ of Hadamard's inequality. A tighter version of Lemma \ref{lm:det_bound} can be obtained at the cost of brevity, by explicitly maximizing with respect to the parameter $x:= \sum_{j=1}^{\ell} \sigma^2_{n-j+1}(A)$ rather than upper bounding both $\sum_{i=1}^{n-\ell} \sigma^2_i(A)$ and $\sum_{j=1}^{\ell} \sigma^2_{n-j+1}(A)$ with $n^2$. However, this optimization does not lead to any improvement in the exponential constant of Theorem \ref{thm:main}, and so its derivation is left to the interested reader.

\subsection{An Improved Optimization Problem}\label{sub:improvedopt}
Lemma \ref{lm:det_bound} applied to the matrix iterates $A^{(k)} \in \mathbb{C}^{k \times k}$ of Gaussian elimination under complete pivoting leads to further constraints on the pivots $p_k = |A^{(k)}|_\infty$. Consider some $0< \ell <k$ with $k + \ell \le n$. Using block notation, let $N^{(1,1)}$, $N^{(1,2)}$, $N^{(2,1)}$, and $N^{(2,2)}$ denote the upper-left $\ell \times \ell$, upper-right $\ell \times k$, lower-left $k \times \ell$, and lower-right $k \times k$ sub-matrices of $A^{(k+\ell)}$. After $\ell$ further steps of Gaussian elimination applied to $A^{(k+\ell)}$, we obtain
\[ A^{(k+\ell)}=\begin{bmatrix} N^{(1,1)} & N^{(1,2)} \\ N^{(2,1)} & N^{(2,2)} \end{bmatrix}  = \begin{bmatrix} \widetilde{L} & 0  \\ N^{(2,1)}\widetilde{U}^{-1} & I  {}\end{bmatrix}\begin{bmatrix} \widetilde{U} & \widetilde{L}^{-1} N^{(1,2)}  \\  0 & N^{(2,2)} - N^{(2,1)} [N^{(1,1)}]^{-1} N^{(1,2)}  \end{bmatrix} ,\]
where $\widetilde{L}  \widetilde{U}$ is the LU factorization of $N^{(1,1)}$, implying that 
\[A^{(k)} =  N^{(2,2)} - N^{(2,1)} [N^{(1,1)}]^{-1} N^{(1,2)}.\]
For the sake of space, let $X := N^{(2,2)}$ and $Y:=N^{(2,1)} [N^{(1,1)}]^{-1} N^{(1,2)}$, and note that $Y$ has rank at most $\ell$. We may rewrite $A^{(k)}$ as
\begin{equation}\label{eqn:split}
A^{(k)} =  \bigg(X - \frac{\text{Re} \langle X,Y \rangle_F}{\|Y\|_F^2} Y \bigg) - \bigg( 1 - \frac{\text{Re} \langle X,Y \rangle_F}{\|Y\|_F^2} \bigg)Y.
\end{equation}
We note that
\[ \bigg\| X - \frac{\text{Re}\langle X, Y\rangle_F}{\|Y\|_F^2} Y \bigg\|_F^2 = \|X\|_F^2 - \frac{\big(\text{Re} \langle X,Y \rangle_F\big)^2}{\|Y\|_F^2} \le \|X\|_F^2 \le p_{k+\ell}^2 n^2 \]
and
\begin{align*}
    \bigg\| \bigg( 1 - \frac{\text{Re}\langle X, Y\rangle_F}{\|Y\|_F^2} \bigg) Y \bigg\|_F^2 &= \|Y\|_F^2 - 2 \,\text{Re} \langle X,Y \rangle_F + \frac{\big(\text{Re} \langle X,Y \rangle_F \big)^2}{\|Y\|_F^2} \\ &\le \|Y\|_F^2 - 2 \, \text{Re} \langle X,Y \rangle_F + \|X\|^2_F \\ &= \| X-Y \|_F^2 \le p_{k}^2 n^2,
\end{align*}
as the entries of $A^{(k)}$ and $N^{(2,2)}$ have modulus at most $p_k$ and $p_{k+\ell}$, respectively. Applying Lemma \ref{lm:det_bound} to $A^{(k)}$ using the splitting in Equation \ref{eqn:split}, we obtain the bound
\begin{equation}\label{ineq:nl}
\frac{\prod_{i=1}^k p_i}{p_{k+\ell}^k} = \frac{\det(A^{(k)})}{p_{k+\ell}^k} \le \frac{k^k}{(k-\ell)^{\frac{k-\ell}{2}}\ell^{\frac{\ell}{2}}} \bigg( 1 + \frac{p_k}{p_{k+\ell}}\bigg)^\ell.
\end{equation}
Making use of these additional constraints gives the following refinement of Optimization Problem \ref{prog:wilko}:

\begin{tcolorbox}
\begin{equation}\label{prog:improvedo}
\begin{gathered}
\textbf{Improved Optimization Problem} 
 \\[1ex]
\begin{array}{cll}
\textrm{max}  & p_1/p_n & \\
\textrm{s.t.} & \prod_{i = 1}^k p_i \le  k^{k/2}p_k^k   &\text{for } k = 1,...,n\\[.25ex]
& \prod_{i = 1}^k p_i \le\displaystyle{ \frac{k^kp_{k+\ell}^{k-\ell}(p_k + p_{k+\ell})^\ell }{(k-\ell)^{\frac{k-\ell}{2}}\ell^{\frac{\ell}{2}}}}&\text{for } \ell = 1,..., \min\{k-1,n-k\} \\[-1.5ex]
& & \quad \; \, k = 2,...,n-1.\\
\end{array}
\end{gathered}
\end{equation}
\end{tcolorbox}
\vspace{1 mm}

\subsection{From a Non-Linear to Linear Program}\label{sub:improvedlp}
The additional constraints given by Inequality \ref{ineq:nl} for $k = 2,...,n-1$ and $\ell = 1,...,\min\{k-1,n-k\}$ produce an optimization problem (Optimization Problem \ref{prog:improvedo}) that is no longer linear upon the transformation $q_k = \ln (p_k)$, $k = 1,...,n$. For this reason, we relax Optimization Problem \ref{prog:improvedo} in order to maintain linearity. For simplicity, we do so while giving only minor attention to lower-order terms (i.e., terms that do not affect the leading exponential constant). More complicated linear programs with improved behavior for finite $n$ can be obtained by a more involved analysis.

Consider an arbitrary feasible point $(p_1,...,p_n)$ of Optimization Problem \ref{prog:improvedo}. We claim that $(p_1,...,p_n)$ also satisfies
\begin{align}\label{ineq:r}
    \prod_{i=1}^k p_i \le (\tfrac{11}{4}k)^{k/2} p_{k+\ell}^{k-\ell} \; p_{k}^\ell \qquad \text{for} \quad &\ell = 1,...,\min\{k-1,n-k\} \\[-2.25ex] &k = 2,...,n-1. \nonumber
\end{align}
We break our analysis into two cases. If $p_k \le (\sqrt{11}/2)^{k/(k-\ell)} p_{k+\ell}$, then 
\[ \prod_{i=1}^k p_i \le k^{k/2} p_k^k  \le (\tfrac{11}{4} k)^{k/2} p_{k+\ell}^{k-\ell} \; p_k^{\ell}.\]
Conversely, if $p_k \ge (\sqrt{11}/2)^{k/(k-\ell)} p_{k+\ell}$, then
\begin{align*} 
\prod_{i=1}^k p_i  &\le \frac{k^k p_{k+\ell}^{k-\ell} \, ( p_{k} + p_{k+\ell} )^\ell}{(k-\ell)^{\frac{k-\ell}{2}}\ell^{\frac{\ell}{2}}}   \le k^{k/2} p_{k+\ell}^{k-\ell} \; p_{k}^\ell \bigg( \frac{k^{k/2} \big(1 + \big(\frac{2}{\sqrt{11}}\big)^{k/(k-\ell) }\big)^\ell}{(k-\ell)^{\frac{k-\ell}{2}}\ell^{\frac{\ell}{2}}}   \bigg) \le (\tfrac{11}{4}k)^{k/2} p_{k+\ell}^{k-\ell} p_{k}^\ell,
\end{align*}
where we have used the fact that
\[ \max_{t \in (0,1)} \bigg(\frac{1}{t} \bigg)^{\frac{t}{2}}\bigg(\frac{1}{1-t}\bigg)^{\frac{1-t}{2}} = \sqrt{2} \quad \text{and} \quad \max_{t \in (0,1)} \bigg(1 + \bigg(\frac{2}{\sqrt{11}}\bigg)^{1/(1-t) }\bigg)^t \approx 1.168 < \sqrt{\frac{11}{8}}.\]
Applying the transformation $q_k = \ln (p_k)$, $k = 1,...,n$, to Inequality \ref{ineq:r}, we obtain the linear program:

\begin{tcolorbox}
\begin{equation}\label{prog:lin}
\begin{gathered}
\textbf{Improved Linear Program}  \\[1ex]
\begin{array}{cll}
\textrm{max}& q_1 - q_n & \\
\textrm{s.t.}  &  \sum_{i = 1}^k q_i \le \frac{k}{2} \ln(k) + k q_k &\text{for } k = 1,...,n  \\
& \sum_{i = 1}^k q_i \le \frac{k}{2} \ln(\tfrac{11}{4}k) + (k-\ell) q_{k+\ell} + \ell q_k &\text{for } \ell = 1,..., \min\{k-1,n-k\} \\
& & \quad \; \, k = 2,...,n-1.\\
\end{array}
\end{gathered}
\end{equation}
\end{tcolorbox}
\vspace{1mm}

\noindent and note that the maximum growth factor $g_n(\mathbb{C})$ is upper bounded by $e^{\mathrm{OPT}}$, where $\mathrm{OPT}$ is the optimal value of this linear program. Program \ref{prog:lin} is an improved version of Wilkinson's linear program (Program \ref{prog:wilkl}), containing all of Wilkinson's constraints as well as additional bounds representing long-range interactions (i.e., bounds relating $A^{(k)}$ and $A^{(k+\ell)}$). In addition, we note that the optimal value of Program \ref{prog:lin} and the logarithm of the optimal value of Program \ref{prog:improvedo} are asymptotically equal up to lower order terms:
\begin{proposition}
If OPT is the optimal value of Linear Program \ref{prog:lin} for $n$, then the optimal value of Optimization Problem \ref{prog:improvedo} for $n$ lies in the interval $[n^{-3/2} e^{OPT},e^{OPT}]$.
\end{proposition}

\begin{figure}
\subfigure[Comparing Wilkinson's Bound to our Improved Bound for $g_n(\mathbb{C})$ for $n \le 5000$]{\includegraphics[width = 7.2cm]{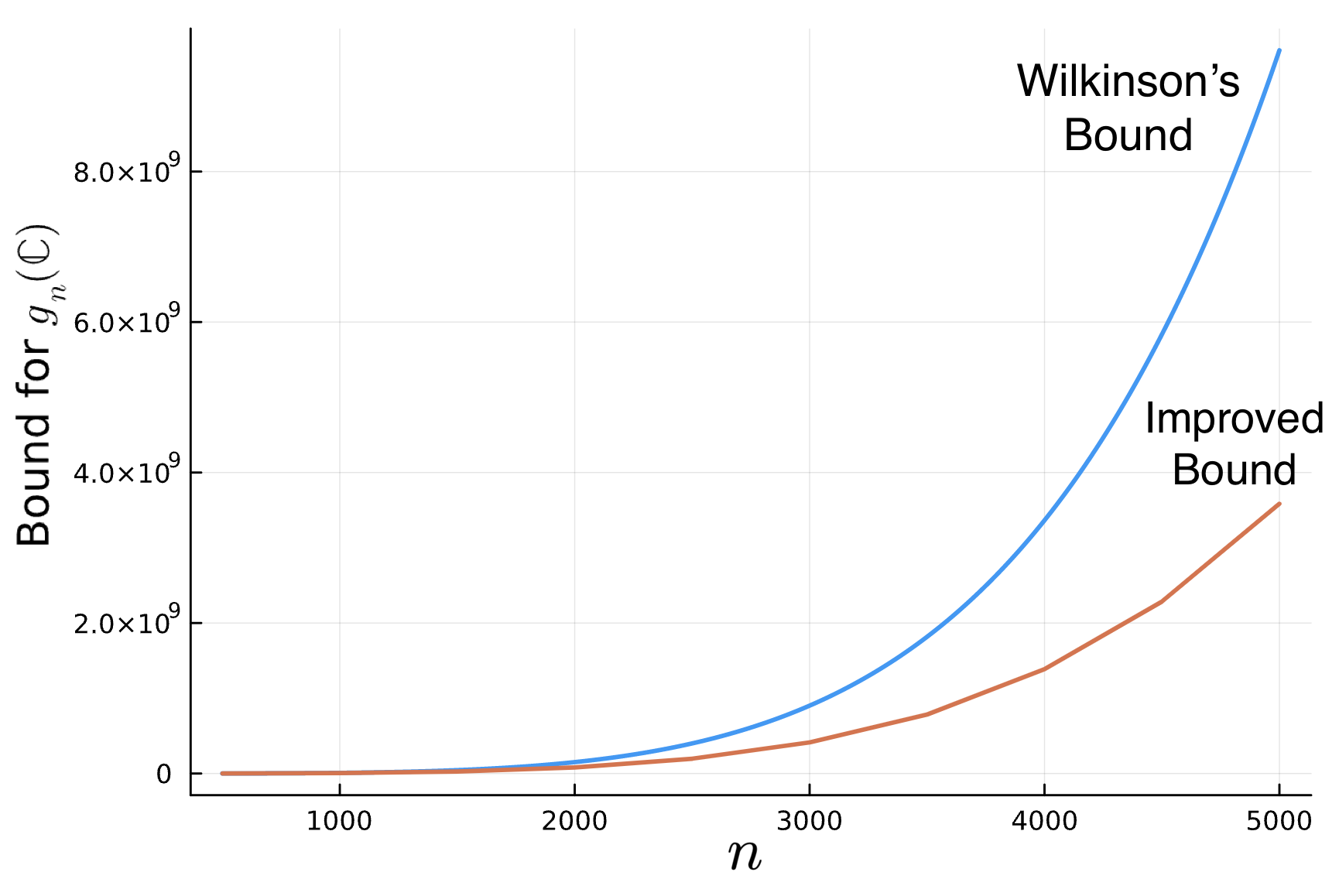}}\qquad
\subfigure[Active constraints parameterized by $(k,\ell)$ for the optimal solution to Program \ref{prog:lin} at $n = 5000$]{\includegraphics[width = 7.2cm]{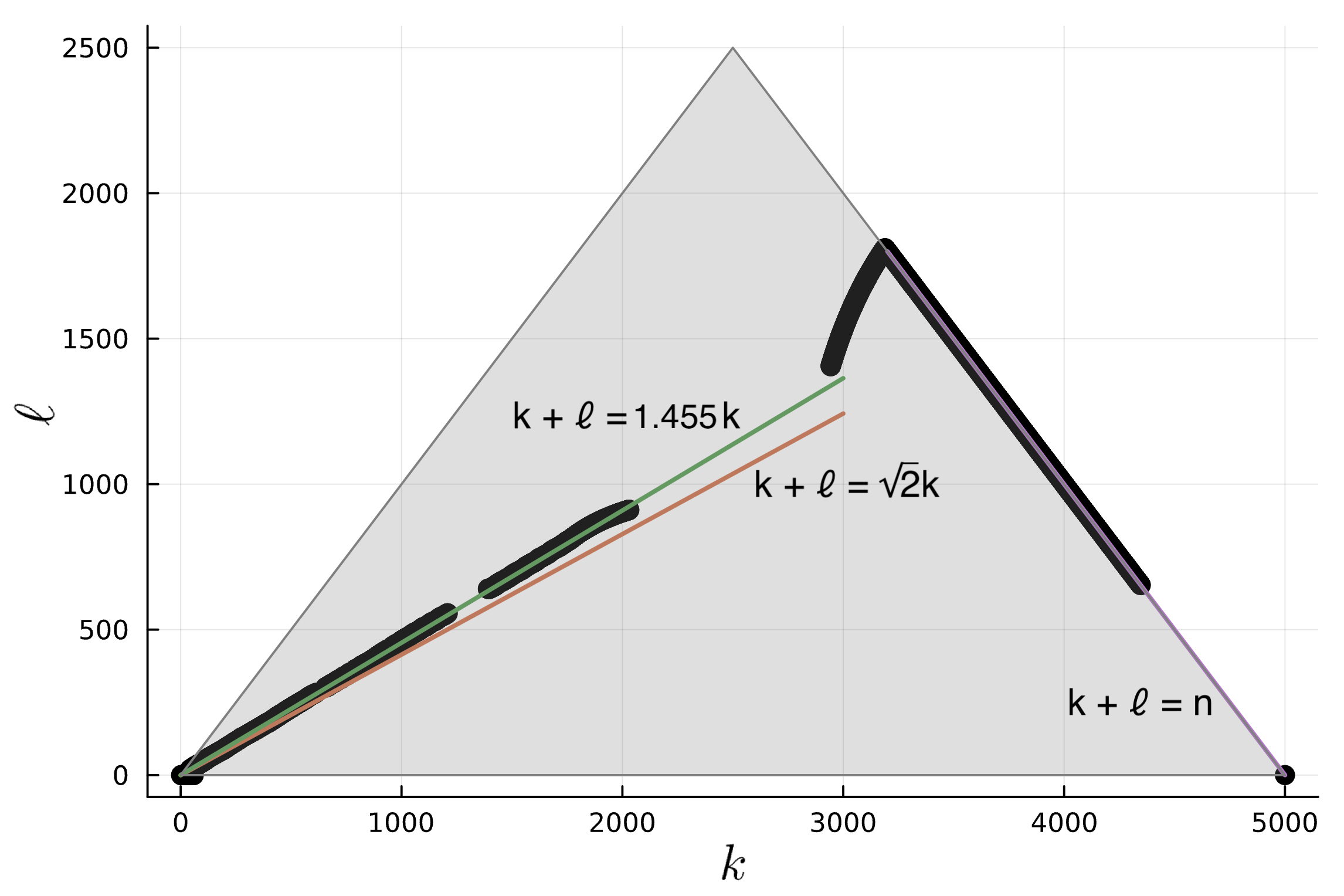}}
\caption{Comparing our Improved Linear Program to Wilkinson's LP: Figure (a) illustrates the difference between Wilkinson's bound for $g_n(\mathbb{C})$ (Inequality \ref{eqn:wilk}) and the upper bound produced by the optimal value of Program \ref{prog:lin} for $n \le 5000$.  Figure (b) is a scatter plot of the pairs $(k,\ell)$ for which the corresponding inequality in Program \ref{prog:lin} is tight for a numerically computed optimal solution at $n = 5000$. The grey shaded triangle shows the set of $(k,\ell)$ corresponding to constraints of Program \ref{prog:lin}, with Wilkinson's constraints parameterized by $(k,0)$, and the black dots represent the subset of those constraints that are active for the numerically computed optimal solution. For $n =5000$, almost none of Wilkinson's constraints are active. The red line $ k + \ell = \sqrt{2}  k$ is the set of constraints used to prove Theorem \ref{thm:main}, and the green line denotes the asymptotically tight constraints for the feasible point produced in Subsection \ref{sub:tight}. While the points on the purple line $k + \ell = n$ improves the objective value, these constraints do not play a role in the asymptotic leading term of the solution to the linear program.
}
\label{fig:lin}
\end{figure}

\begin{proof}
Let $(q_1,...,q_n)$ be a feasible point of Linear Program \ref{prog:lin}. It suffices to show that $p_k = k^{3/2} e^{q_k}$, $k = 1,...,n$, is a feasible point of Optimization Problem \ref{prog:improvedo}. Considering an arbitrary constraint parameterized by $k > 1$ and $\ell$, we have
\[\prod_{i=1}^k p_i = (k!)^{\tfrac{3}{2}} \exp\bigg\{\sum_{i=1}^k q_i\bigg\} \le (k!)^{\tfrac{3}{2} }\exp\bigg\{\frac{k}{2} \ln(\tfrac{11}{4}k) + (k-\ell) q_{k+\ell} + \ell q_k\bigg\} .\]
Rewriting the right-hand side in terms of $p_k$ gives
\begin{align*}
\prod_{i=1}^k p_i  &\le \frac{(k!)^{\tfrac{3}{2}}(\tfrac{11}{4}k)^{\tfrac{k}{2}} p_{k+\ell}^{k-\ell} \,p_k^\ell }{(k+\ell)^{\tfrac{3}{2}(k-\ell)} k^{\tfrac{3}{2}\ell}} \le k^{\tfrac{k}{2}} p_{k+\ell}^{k-\ell} \,p_k^\ell \frac{(k!)^{\tfrac{3}{2}}(\tfrac{11}{4})^{\tfrac{k}{2}} }{k^{\tfrac{3}{2}k}} \le k^{\tfrac{k}{2}} p_{k+\ell}^{k-\ell} \,p_k^\ell \le \frac{k^kp_{k+\ell}^{k-\ell}(p_k + p_{k+\ell})^\ell }{(k-\ell)^{\frac{k-\ell}{2}}\ell^{\frac{\ell}{2}}},
\end{align*}
completing the proof.
\end{proof}

In the following section, we provide nearly matching upper and lower bounds on the optimal value of Program \ref{prog:lin} for sufficiently large $n$, thereby proving Theorem \ref{thm:main}.

\subsection{Bounding the Growth Factor in Practice}\label{sub:computation} While the proof of Theorem \ref{thm:main} focuses on the behavior for large $n$, we note that an improvement in exponential constant exists in practice for reasonably sized matrices as well. We provide a comparison of the optimal value of Program \ref{prog:lin} to the optimal value of Wilkinson's LP in Figure \ref{fig:lin} for $n \le 5000$. The numerically computed solutions to Program \ref{prog:lin} were obtained using the Gurobi Optimizer \cite{gurobi} called through the JuMP package for mathematical optimization \cite{lubin2023} in the Julia programming language \cite{bezanson2017julia}. We stress that numerically computed solutions to a linear program can be converted into mathematical bounds via a dual feasible point verified in exact arithmetic. In addition, Program \ref{prog:lin} can be adapted in a number of ways for computational efficiency. For instance, the linear transformation $Q(k)= \sum_{i=1}^k q_i$ produces a linear program with a simple objective and sparse constraints (at most four variables in each). Furthermore, as the analysis in Section \ref{sec:lpbound} suggests, only a linear number of constraints are required to produce a reasonable upper bound for the optimal value. One natural choice consists of Wilkinson's original constraints and additional constraints of the form $k+ \ell = n$ and $k+ \ell \in[\sqrt{2}k-1, \sqrt{2}k + C]$ for some constant $C$ (Theorem \ref{thm:main} is proved using only constraints of the form $k+ \ell = \lceil \sqrt{2}k \rceil$). Finally, we stress that the techniques used to produce improved estimates can be further optimized to obtain even better bounds in both theory and practice. We hope that the interested reader will do so.

\section{Bounding the Optimal Value of our Linear Program}\label{sec:lpbound}

Finally, we prove that the objective of Program \ref{prog:lin} satisfies the bound
\[\max \; q_1-q_n \, \leq \,  \alpha \ln^2 n +(\beta + 1/2)\ln n, \qquad \text{where} \quad \alpha = \frac{1}{2(2+(2-\sqrt{2}) \ln 2)}\]
and $\beta = 0.41$,
thus completing the proof of Theorem \ref{thm:main} ($\beta = 0.41$ corresponds to the constant $\beta + 1/2 = 0.91$ in Theorem \ref{thm:main}). We do so via a duality argument, making use of the constraints for $k$ and $\ell$ satisfying $k + \ell \approx \sqrt{2} k$. Before proving the above bound, we first illustrate why $[2(2+(2-\sqrt{2}) \ln 2)]^{-1}$ is the correct choice of $\alpha$ for constraints of the form $k + \ell \approx \sqrt{2} k$, and show that this choice is within $0.00024$ of the exact asymptotic constant of Program \ref{prog:lin}.

\subsection{On the Choice and Optimality of the Constant $\alpha=[2(2+(2-\sqrt{2}) \ln 2)]^{-1}$}\label{sub:tight}
Suppose that $q_x-q_1 = - \gamma \ln^2 x+O(1)$. Then, for the constraint 
\[ \sum_{i=1}^k(q_i - q_1) \le \frac{k}{2}\ln (\tfrac{11}{4} k) + (k-\ell)(q_{k+\ell} -q_1) + \ell (q_k - q_1),\]
the left-hand side equals
\[\int_1^k -\gamma \ln^2 x \; \mathrm{d} x+O(k)=-\gamma k \ln^2 k + 2\gamma k \ln k + O(k) \]
and the right-hand side equals 
\[- \gamma k \ln^2 k + \big[k/2 - 2 \gamma (k-\ell) \ln(1+ \ell/k)\big] \ln k + O(k).\]
Letting $t = \ell/k$, the right-hand side is asymptotically larger than the left-hand side if
\[ \gamma \leq \frac{1}{4(1+(1-t)\ln(1+t))}.\]
The values $t = 0$ and $t=1$ (i.e., when $\ell = 0$ or $\ell = k$) correspond to the constraints of Wilkinson's linear program, and for $t = 0$ and $t =1$, we obtain $\gamma \le 1/4$ (e.g., Wilkinson's bound). The value $t = \sqrt{2} -1$ produces the upper bound $1/[2(2+(2-\sqrt{2}) \ln 2)]\approx 0.20781$ of Theorem \ref{thm:main}. The quantity $[4(1+(1-t)\log(1+t))]^{-1}$ on the interval $[0,1]$ is minimized by $t = \exp\{W(2e) - 1\} -1 \approx 0.4547$, where $W(x)$ is the Lambert W function, with a minimum value of 
\[  \frac{1}{4\big(1+ (2-e^{W(2e)-1})(W(2e)-1)\big)} \approx 0.207576.\]
This implies the existence of a solution to Program \ref{prog:lin} with $q_1-q_n = 0.207575 \ln^2 n - O(\ln n)$, thus illustrating that our upper bound of $\alpha=[2(2+(2-\sqrt{2}) \ln 2)]^{-1}\approx 0.207811$ is within $0.00024$ of the optimal value of the linear program. We do not pursue further improvement on this constant.

\subsection{Reducing Theorem \ref{thm:main} to Geometric Mean Growth}
For ease of analysis, we consider a continuous version of our variables $q = (q_1,...,q_n)$. Let 
$$ f(x) = q_{\lceil x \rceil}-q_1 \qquad \text{and} \qquad F(x)= \frac{1}{x} \int_0^x f(t) \, \mathrm{d}t \qquad \text{for } x>0,$$
where $\{q_k\}_{k=1}^\infty$ is any sequence such that $q_1\ge q_k$ for all $k \in \mathbb{N}$ and $(q_1,...,q_n)$ is a feasible point of Program \ref{prog:lin} for all $n \in \mathbb{N}$. $F(x)$ can be thought of a continuous version of geometric mean growth (described in Section \ref{sec:wilk}). Any optimal solution $(q_1,...,q_n)$ for the $n$-dimensional linear program can be converted into such a sequence by simply setting $q_{k} = q_n$ for all $k >n$. The constraint of Program \ref{prog:lin} with $k=\lceil x \rceil$ and $\ell=\lceil\sqrt{2} x \rceil-\lceil x \rceil$ implies that for all $x>0$,
\begin{align}F(\lceil x \rceil) &\leq \frac{\ln (\tfrac{11}{4}\lceil x \rceil)}{2} + \left( \frac{2 \lceil x \rceil - \lceil \sqrt{2} x \rceil}{\lceil x \rceil} \right)f(\sqrt{2}x)+\left( \frac{\lceil \sqrt{2} x \rceil-\lceil x \rceil}{\lceil x \rceil} \right)f(x) \nonumber \\
 &\leq \frac{\ln (\tfrac{11}{4} x)}{2} + \frac{1}{2x} +\left( \sqrt{2} - 1 - \frac{\sqrt{2}}{x} \right)\left( \sqrt{2}f(\sqrt{2}x) + f(x) \right).\label{ineq:f}
\end{align}
We make the following claim regarding $F(x)$ (recall, $\alpha=[2(2+(2-\sqrt{2}) \ln 2)]^{-1}$ and $\beta = 0.41$).
\begin{lemma}\label{claim}
$F(x) > -\alpha \ln^2 x - \beta \ln x$ for all $x > 100$.
\end{lemma}
Lemma \ref{claim} implies our desired result, as
$$F(n) = \frac{1}{n}  \sum_{i=1}^n (q_i-q_1) \le \frac{1}{n}\bigg( \frac{n}{2} \ln n + n q_n - n q_1\bigg),$$
and $\alpha \ln^2 n + (\beta + 1/2) \ln n$ is larger than Wilkinson's bound for $x \le 100$. A tighter bound may be obtained by adding together constraints of the form $k + \ell = n$ for $k \ge n/(8\alpha)$ (e.g., the constraints appearing in Figure \ref{fig:lin}(b)). However, the analysis is involved and the improvement on the $1/2 \, \ln n$ term produced by the argument above is minor ($\approx 0.046$ improvement, at the cost of lower-order terms).

\subsection{Proof of Lemma \ref{claim}: Base Case}
The proof of Lemma \ref{claim} is, in spirit, by ``induction on $x$" via a duality argument. Clearly the assertion holds for $x \in (100,1700]$ for $\beta$ sufficiently large. However, verifying the base case of $x \in (100,1700]$ for $\beta = 0.41$ requires some analysis, as the quantity $\alpha \ln^2 n + \beta \ln n$ is strictly less than Wilkinson's bound. We have
\[ F(x) = \frac{1}{x} \int_0^x q_{\lceil t \rceil} - q_1 \,  \mathrm{d}t = \frac{x - \lfloor x \rfloor}{x} (q_{\lceil x \rceil} -q_1) + \frac{1}{x} \sum_{k =1}^{\lfloor x \rfloor} (q_k - q_1)  .\]
By Inequalities \ref{eqn:wilk} and \ref{ineq:avg_bound}, 
\[q_1 - q_{\lceil x \rceil} \le \frac{\ln^2 \lceil x \rceil}{4} + \frac{\ln \lceil x \rceil}{2} + \ln 2 \qquad \text{ and } \qquad \frac{1}{\lfloor x \rfloor} \sum_{k = 1}^{\lfloor x \rfloor} (q_1 - q_k) \le \frac{\ln^2 \lfloor x \rfloor}{4} + \ln 2.\] 
Altogether, we obtain the lower bound
\begin{align*}
    F(x) &\ge - \frac{1}{x}\bigg(\frac{\ln^2 \lceil x \rceil}{4} + \frac{\ln \lceil x \rceil}{2} + \ln 2\bigg) - \bigg(\frac{\ln^2 \lfloor x \rfloor}{4} + \ln 2\bigg)\\
    &\ge - \frac{1}{x}\bigg(\frac{(\ln x  + \tfrac{1}{x})^2}{4} + \frac{\ln x  + \tfrac{1}{x}}{2} + \ln 2\bigg) - \bigg(\frac{\ln^2  x }{4} + \ln 2\bigg).
\end{align*} 
By inspection, the right-hand side of the above inequality is strictly greater than $-(\alpha \ln^2 x + \beta \ln x)$ for our interval of interest $x \in [100,1700]$.

\subsection{Proof of Lemma \ref{claim}: Inductive Step}
In order to verify the claim for some $y>1700$, we integrate over $x \in \big[\tfrac{y}{2},\tfrac{y}{\sqrt{2}}\big]$ to obtain a lower bound for $F(y)$ in terms of $F(x)$ for $x <y$. In particular, by integrating Inequality \ref{ineq:f} over $x \in \big[\tfrac{y}{2},\tfrac{y}{\sqrt{2}}\big]$ we have
\begin{align*}
\frac{1}{\frac{y}{\sqrt{2}} - \frac{y}{2}} \int_{\frac{y}{2}}^{\frac{y}{\sqrt{2}}} F(\lceil x \rceil) \, \mathrm{d}x   &\leq  \frac{1}{\frac{y}{\sqrt{2}} - \frac{y}{2}} \left[ \bigg(\sqrt{2} - 1 - \frac{2\sqrt{2}}{y}\bigg) \int_{\frac{y}{2}}^{y} f(x) \, \mathrm{d}x + \int_{\frac{y}{2}}^{\frac{y}{\sqrt{2}}} \frac{\ln(\tfrac{11}{4}x)}{2} + \frac{1}{2x} \, \mathrm{d}x   \right]\\
 &= \bigg(1- \frac{4+2\sqrt{2}}{y}\bigg) \big( 2 F(y) - F(\tfrac{y}{2}) \big) + \frac{\ln y}{2}+ \frac{\ln 2 + \sqrt{2} \ln\tfrac{11}{4} - \sqrt{2} }{2\sqrt{2}} + \frac{(\sqrt{2}+1)\ln 2}{2y}.   
\end{align*}
Rearranging the above inequality allows us to lower bound $F(y)$ by a positive linear combination of $F(x)$ for $x \in \big[\tfrac{y}{2},\tfrac{y}{\sqrt{2}}\big]$. We note that this is the reason for the choice of $k + \ell \approx \sqrt{2} k$, as this approach does not give us such a bound if $\sqrt{2}$ is replaced by a larger constant. Now, suppose our claim is false, and let $y > 1700$ be the smallest value such that $F(y) \le -\alpha \ln^2 y - \beta \ln y$. We aim to show that this contradicts the above lower bound for $F(y)$. By assumption,
\begin{align*}
    F(\lceil x \rceil) &>  -\alpha \ln^2(x+1) -\beta \ln (x+1)  \\
    &>-\alpha \ln^2 x -\beta \ln x   - \frac{2 \alpha \ln x}{x} - \frac{\beta}{x} - \frac{\alpha}{x^2} \qquad \text{for } x \in \big[\tfrac{y}{2} , \tfrac{y}{\sqrt{2}}\big],
\end{align*} 
implying that
\begin{align*}
\frac{1}{\frac{y}{\sqrt{2}} - \frac{y}{2}} \int_{\frac{y}{2}}^{\frac{y}{\sqrt{2}}} F(\lceil x \rceil) \, \mathrm{d}x   &>  - \alpha \ln^2 y - \big((\sqrt{2}\ln 2 -2)\alpha +\beta \big) \ln y -\bigg(2  - \frac{(3+\sqrt{2}) \ln^2 2}{2 \sqrt{2}} -\sqrt{2}\ln 2\bigg)\alpha \\ &\qquad - \bigg(\frac{\ln 2}{\sqrt{2}} -1 \bigg) \beta   - \frac{2 (\sqrt{2}+1) \alpha \ln 2  \ln y}{y}   - \frac{(\sqrt{2}+1)(  \beta \ln 2  - \tfrac{3}{2}\alpha \ln^2 2  )}{y} - \frac{2 \sqrt{2} \alpha}{y^2}.
\end{align*}
In addition,
\[ 2F(y) - F(\tfrac{y}{2})  <  - \alpha \ln^2 y - (2 \alpha \ln 2 + \beta) \ln y + \alpha \ln^2 2 - \beta \ln 2. \]
Combining our upper and lower bounds, we observe that the terms containing $\ln^2 y$ are equal, and the terms containing $\ln y$ are equal
$$ - \big((\sqrt{2}\ln 2 -2)\alpha +\beta \big) = \frac{1}{2} - (2 \alpha \ln 2 + \beta)$$
due to the value of $\alpha$. We are left with the inequality
\[ \frac{(\sqrt{2}-1)\ln 2 + \sqrt{2}}{\sqrt{2}} \beta + \frac{(2-\sqrt{2})\ln^2 2 - 4(2-\sqrt{2}) ( \ln \frac{11}{4} - 1)\ln 2  - 8 \ln \frac{11}{4} }{8(2+(2-\sqrt{2}) \ln 2)}  + g(\beta,y)<0,\]
where $g(\beta,y)$ is a linear function of $\beta$ of order $O(\ln^2 (y)/y)$. The left-hand side is strictly greater than zero for a sufficiently large choice of $\beta$. However, verifying that our choice of $\beta = 0.41$ is sufficient requires an explicit analysis of $g(\beta,y)$ for $\beta = 0.41$ and $y >1700$. The function $g(\beta,y)$ is given by
\begin{align*}
g(\beta,y) &= - \frac{2+\sqrt{2}}{2+(2-\sqrt{2})\ln 2} \, \frac{\ln^2 y}{y} - \bigg( (4+2\sqrt{2}) \beta + \frac{(5 + 3\sqrt{2}) \ln 2}{2+(2-\sqrt{2})\ln 2}\bigg) \, \frac{\ln y}{y}\\
&\qquad + \bigg( \frac{(11+7 \sqrt{2})\ln^2 2 }{4(2+(2-\sqrt{2})\ln 2)} -  (5 + 3\sqrt{2}) \beta \ln 2 -\frac{(\sqrt{2}+1)\ln 2}{2}  \bigg) \,\frac{1}{y}  - \frac{\sqrt{2}}{2+(2-\sqrt{2})\ln 2} \, \frac{1}{y^2}.
\end{align*}
When $\beta = 0.41$ and $y >1700$,
\[\frac{(\sqrt{2}-1)\ln 2 + \sqrt{2}}{\sqrt{2}} \beta + \frac{(2-\sqrt{2})\ln^2 2 - 4(2-\sqrt{2}) ( \ln \frac{11}{4} - 1)\ln 2  - 8 \ln \frac{11}{4} }{8(2+(2-\sqrt{2}) \ln 2)} >0.086 \]
and
\[g(0.41,y) > - \frac{\tfrac{3}{2}\ln^2 y}{y} -  \frac{6\ln y}{y}- \frac{3}{y} - \frac{1}{y^2} >- \frac{\tfrac{3}{2}\ln^2 1700}{1700} -  \frac{6\ln 1700}{1700}- \frac{3}{1700}  -\frac{1}{1700^2}> -0.08,\]
thus obtaining our desired contradiction. This completes the proof of Theorem \ref{thm:main}.

\section*{Acknowledgements}

\setstretch{.4}
\begin{onehalfspacing}

{\tiny This material is based upon work supported by the National Science Foundation under grant no. OAC-1835443, grant no. SII-2029670, grant no. ECCS-2029670, grant no. OAC-2103804, and grant no. PHY-2021825. We also gratefully acknowledge the U.S. Agency for International Development through Penn State for grant no. S002283-USAID. The information, data, or work presented herein was funded in part by the Advanced Research Projects Agency-Energy (ARPA-E), U.S. Department of Energy, under Award Number DE-AR0001211 and DE-AR0001222. The views and opinions of authors expressed herein do not necessarily state or reflect those of the United States Government or any agency thereof. This material was supported by The Research Council of Norway and Equinor ASA through Research Council project ``308817 - Digital wells for optimal production and drainage''. Research was sponsored by the United States Air Force Research Laboratory and the United States Air Force Artificial Intelligence Accelerator and was accomplished under Cooperative Agreement Number FA8750-19-2-1000. The views and conclusions contained in this document are those of the authors and should not be interpreted as representing the official policies, either expressed or implied, of the United States Air Force or the U.S. Government. The U.S. Government is authorized to reproduce and distribute reprints for Government purposes notwithstanding any copyright notation herein. The third author thanks Mehtaab Sawhney for interesting conversions regarding linear programming. The authors thank Louisa Thomas for improving the style of presentation.}
\end{onehalfspacing}
{ \small 
	\bibliographystyle{plain}
	\bibliography{main.bib} }
 
\newpage
\appendix

\section{The modern role of partial and complete pivoting in computation}\label{app:a}

This appendix reviews the role of the growth factor in applied computation and demonstrates the modern practical importance of both partial and complete pivoting.

\subsection{The growth factor and error estimates for solving linear systems} 
Gaussian elimination can be used to solve a linear system $A x =b$ by factoring $A = LU$ into the product of a lower triangular and upper triangular matrix $L$ and $U$. Given the factorization $A = LU$, the linear system $Ax = b$ is mathematically equivalent to $LU x = b$, which can be efficiently and accurately solved using forward and backward substitution. This procedure, when performed in floating point arithmetic with either partial or complete pivoting, produces an approximate solution $\hat x$ satisfying 
\[(A+\Delta A)\hat x = b, \qquad \|\Delta A \|_{\max} \le \tfrac{2}{1-nu}n^3 \mathrm{u} \, \rho(A),\]
where $n$ is the dimension of the matrix, $\mathrm{u}$ is the unit roundoff of the floating point arithmetic, and $\rho(A)$ is the growth factor of $A$ in floating point arithmetic under pivoting (the same bound holds for both partial and complete pivoting); see \cite[pgs. 175-177]{higham2002accuracy}  and other related formulas for details. For this reason, understanding the growth factor is of both great practical and theoretical importance. Further details regarding the long history of research in this area can be found in \cite{edelman2024some}, though we also draw special attention to the modern interest in smoothed analysis (the study of algorithms under small random perturbation to input) for Gaussian elimination \cite{sankar2004smoothed,sankar2006smoothed}.

\subsection{Is large growth for partial pivoting as rare now as it seemed in years past?}
In the classic text {\it The Algebraic Eigenvalue Problem} \cite[pg. 212]{Wilkinson1965AEP}, Wilkinson showed that Gaussian elimination with partial pivoting was unstable in the worst case. He proved that, for partial pivoting, the growth factor is bounded above by $2^{n-1}$, and that this exponential upper bound can be achieved by the matrix
\begin{equation}\label{eqn:wilk_matrix}
A = \begin{pmatrix*}[r] 1 & 0\phantom{.} & \cdots & 0 & 1 \\ -1 & \ddots & \ddots & \vdots & \vdots \\
\vdots & \ddots & 1 &  0 & 1 \\[1ex]
-1 & \cdots & -1 & 1 & 1 \\[1 ex]
-1 & \cdots & -1 & -1 & 1 \end{pmatrix*}.
\end{equation}
This exponentially large growth factor can lead to catastrophic errors, even for well-conditioned matrices. Despite this, Gaussian elimination with partial pivoting remains the premier technique used to solve a general linear system $Ax=b$ computationally. The backslash ``$\backslash$" command in MATLAB and Julia, and the linalg.solve() command in Python, all employ Gaussian elimination with partial pivoting by calling the same LAPACK routines
when faced with an arbitrary
square
matrix (different algorithms may be used when the input matrix has special structure). However, the fundamental issues originally noted by Wilkinson in 1965 still persist. In Figure \ref{fig:bad}, we attempt to solve a $100 \times 100$ linear system involving the matrix $A$ of Equation \ref{eqn:wilk_matrix} using the built-in backslash ``$\backslash$" command in the Julia programming language. Shockingly, the algorithm produces a solution with almost no correct significant digits, and no error message is output  (see \cite[Subsection 2.2]{peca2023growth} for further experiments). This is unrelated to condition number, as the condition number of $A$ when $n = 100$ is only $45$. Many researchers suspect that the situations where Gaussian elimination with partial pivoting fails (such as Figure \ref{fig:bad}) are rare. For instance, in the widely used textbook {\it Matrix Computations} by Gene Golub and Charles Van Loan, the authors explicitly discuss the worst-case stability of Gaussian elimination with partial pivoting:
\begin{quote}
Although there is still more to understand about $\rho$ [the growth factor], the consensus is that serious element growth in Gaussian elimination with partial pivoting is {\it extremely} rare. {\it The method can be used with confidence} \cite[pg. 131]{golub2013matrix}.
\end{quote}

This sentiment may have been partially true in 1983, when the first edition of {\it Matrix Computations} was released, but it is simply not true today. This could even be considered an example
of normalcy bias, the refusal to plan for a disaster which has never happened before. We contend
that there has been an exponential increase in both the quantity and types of linear systems solved today as compared to when that word
``rare" was typed (likely on an IBM Selectric Typewriter) in the years before 1983.
Furthermore, the current software environments and hardware platforms were likely unimaginable then, especially when considering the 1977 (mis)quote that nobody would ever need a computer in their home.

In the early days of numerical linear algebra, the source of most numerical analysis problems 
were discretizations of integral and differential equations. The matrix $A$ above and matrices with similarly large growth may seem unlikely to occur in practice when considering typical problems of classical numerical analysis.  
Wright \cite{wright1993collection}  and Foster \cite{Foster:1994:GEP}
found examples of two-point boundary value problems and Volterra integral equations with large growth that at least had the appearance of a classical numerical analysis problem. Still, it was easy to argue that even these examples were somewhat contrived.
Nonetheless, given the aforementioned increase in the quantity of linear systems solved, we would
not rule out large growth even amongst classical-looking numerical analysis problems.

The strongest concern, however, arises from the fact that the types of problems solved today are much more varied. The matrix $A$ above (and similar-looking matrices with exponential growth) have a high degree of symmetry and a simple combinatorial structure. The field of discrete mathematics has grown dramatically in recent decades, and with it the need to solve linear systems arising from network structures. As a result, such matrices no longer seem so unlikely to occur in practice\footnote{The third author thanks Avi Widgerson for emphasizing this point during a tutorial on Gaussian elimination presented at the Institute for Advanced Study.}. 

\begin{figure}
\noindent
\begin{quote}
\begin{quote}
\begin{minted}[linenos,autogobble]{julia}
# Wilkinson's famous matrix (in Julia)
# Matrices like wilk(n) should no longer be considered rare 
# (Subsection A.2)
wilk(n) = [ i>j ? -1 : (i==j || j==n) ? 1 : 0 for i=1:n, j=1:n]

# Demonstrating the inaccuracy of GE with partial pivoting
n = 100;
A = wilk(n);
x = randn(n); b=A*x;
[A\b x][45:55,:] # interesting middle elements

#   A\b   vs    x
  0.369141   0.370758
 -0.789062  -0.787242
 -0.835938  -0.82985
  0.78125    0.790538
 -0.875     -0.86182
  0.5625     0.595671
 -3.0625    -2.99157
  1.25       1.33367
  0.0        0.302893
  1.0        1.52312
 -1.0        0.21591
\end{minted}
\end{quote}
\end{quote}
\caption{The potential inaccuracies of partial pivoting. Surprisingly, the computed solution barely has any correct significant digits! Observing the middle elements of the exact and computed solution, one can almost feel the bits being chopped off at the end. This is not caused by the condition number, as cond($A$) is only $45$. No warning or error message is given.}
\label{fig:bad}
\end{figure}

\subsection{Consequences of modern trends in computing} Not only is the solution of $Ax=b$ more frequent and more varied, but
the problem sizes are larger and the types of hardware being used are more varied as well. It is not unusual to take advantage of hardware accelerators such as graphical processing units (GPUs),
which run at lightning speed in half precision. These two factors put further pressure on the accuracy of Gaussian elimination with partial pivoting. 

\subsection{Has complete pivoting software really been out there?}
While LAPACK's partial pivoting routine \verb+getrf+ 
and variants are the workhorse under the hood
of Julia, MATLAB, and Python, LAPACK\footnote{Jim Demmel, one of the lead authors of the LAPACK library, opted for complete pivoting over partial pivoting in his analysis of high accuracy singular value decompositions \cite{demmel1997accurate}.} has a complete pivoting routine available
\verb+getc2+,
 and complete pivoting is also available, though arguably less prominently, in other languages as native code (e.g., on MATLAB Central or in the Julia package DLA.jl). Soon, complete pivoting will become even more accessible in Julia.  The implication that a lack of accessibility of complete
pivoting in software is to be equated with a lack of user interest
is an example of confirmation bias:  a software writer that has solvers embedded in a popular package used by many people may choose the extra safety of complete pivoting if it were more accessible.  \\

Overall, our conclusion
is that stable alternatives to naive partial pivoting are needed and, therefore, a deeper mathematical understanding of both complete and partial pivoting is as important as ever.

\end{document}